\documentclass[12pt,twoside]{amsart}
\usepackage{hyperref}
\usepackage{amssymb}
\usepackage{verbatim}
\usepackage{amsmath}
\usepackage{comment}
\usepackage{bm}
\usepackage{a4wide}
\usepackage[latin1]{inputenc}
\usepackage[T1]{fontenc}
\usepackage{times}
\usepackage{amssymb,latexsym}
\usepackage{enumerate}

\makeatletter
\newcommand{\sumprime}{\if@display\sideset{}{'}\sum%
            \else\sum'\fi}
\makeatother

\begin{document}

\numberwithin{equation}{section}

% define theorem environments
\newtheorem{theorem}{Theorem}[section]
\newtheorem{proposition}[theorem]{Proposition}
\newtheorem{conjecture}[theorem]{Conjecture}
\def\theconjecture{\unskip}
\newtheorem{corollary}[theorem]{Corollary}
\newtheorem{lemma}[theorem]{Lemma}
\newtheorem{observation}[theorem]{Observation}
\newtheorem{definition}[theorem]{Definition}

\thanks{The first author is supported by NSF Grant 11771089 and Gaofeng grant  from School of Mathematical Sciences, Fudan University}

\address{Department of Mathematical Sciences, Fudan University, Shanghai, 20043, China}

\email{boychen@fudan.edu.cn}

\address{Department of Mathematical Sciences, Norwegian University of Science and Technology No-7491 Trondheim, Norway}

\email{xu.wang@ntnu.no}

\title[Bergman kernels]{Bergman kernel and oscillation theory of plurisubharmonic functions}

 \author{Bo-Yong Chen and Xu Wang}
\date{\today}

\begin{abstract} Based on Harnack's inequality and  convex analysis we  show that each plurisubharmonic function is locally BUO (bounded upper oscillation) with respect to  polydiscs of finite type but not for arbitrary polydiscs. We also show that each function in the Lelong class is globally BUO with respect to all polydiscs. A dimension-free BUO estimate is obtained for the logarithm of the modulus of a complex polynomial.  As an application we obtain an approximation formula for the Bergman kernel   that preserves all directional Lelong numbers. For smooth plurisubharmonic functions we derive a new asymptotic identity for the Bergman kernel  from  Berndtsson's complex Brunn--Minkowski theory, which also yields a slightly better version of the sharp Ohsawa--Takegoshi extension theorem in some special cases.

\bigskip


\smallskip

\end{abstract}
\maketitle

\section{Introduction}

Let $\Omega$ be a domain in $\mathbb C^n$ and $PSH(\Omega)$ the set of plurisubharmonic (psh) functions on $\Omega$.  Recall that each $\phi\in PSH(\Omega)$ satisfies the following mean-value inequality:
$$
\phi(z)\le \frac1{|S|} \int_S \phi=:\phi_S
$$
whenever $S$ is a ball or a polydisc, with center $z$. Here $|S|$ denotes the Lebesgue measure of $S$ and $\int_S$ means the Lebesgue integral. The above inequality implies  $\phi\in L^1_{\rm loc}(\Omega)$ and  suggests to estimate the difference $|\phi-\phi_S|$. The concept of BMO functions then enters naturally. Let ${\mathcal S}={\mathcal S}(\Omega)$ be a family of relatively compact open subsets in $\Omega$.
 We say that $\phi\in L^1_{\rm loc}(\Omega)$ has \emph{bounded mean oscillation} (BMO) with respect to ${\mathcal S}$ if
$$
{\sup}_{S\in {\mathcal S}}\,  MO_{S}(\phi) <\infty,\ \ \  MO_{S}(\phi):=\frac1{|S|} \int_S |\phi-\phi_S|.
$$
Let $BMO(\Omega,{\mathcal S})$ denote the set of functions which are BMO with respect to ${\mathcal S}$. When ${\mathcal S}$ is the set of balls in $\Omega$, this is  the original definition of BMO functions due to John-Nirenberg \cite{JN}. A classical example of BMO functions is $\log |z|$. It is also convenient to introduce local BMO functions as follows. For an open set $\Omega_0\subset\subset \Omega$ we define ${\mathcal S}|_{\Omega_0}$ to be the sets of all $S\in {\mathcal S}$ which are relatively compact in $\Omega_0$. Let $BMO_{\rm loc}(\Omega,{\mathcal S})$ be the set of functions on $\Omega$ which belong to $BMO(\Omega_0,{\mathcal S}|_{\Omega_0})$ for every  open set $\Omega_0\subset\subset \Omega$.

By using pluripotential theory, Brudnyi \cite{Brudnyi} was able to show that each psh function is locally BMO with respect to balls (see also \cite{Brudnyi1} for stronger results concerning subharmonic functions in the plane).  Recently, the first author found another approach to local BMO properties of psh functions by using the Riesz decomposition theorem and some basic facts of psh functions (cf. \cite{Chen}). Benelkourchi et al. \cite{BJZ} showed that every function in the Lelong class ${\mathcal L}$ is globally BMO with respect to balls. Recall that
$$
{\mathcal L}=\left\{u\in PSH(\mathbb C^n): {\limsup}_{|z|\rightarrow \infty}\, (u(z)-\log |z|)<\infty\right\}.
$$

In this paper we propose a new and simpler approach based on the following basic observation:

\smallskip

\emph{It is easier to look at the upper oscillation instead of the mean oscillation for psh functions.}

\smallskip

To define the \emph{upper oscillation} one simply uses  $\sup_S \phi$ instead of  $\phi_S$:
\begin{equation}\label{eq:2}
UO_{S}(\phi):=\frac1{|S|} \int_S \left|\phi-{\sup}_S\, \phi\right|={\sup}_S\, \phi-\phi_S.
\end{equation}
Note that $-UO_S(-\phi)$ is exactly  the lower oscillation introduced by Coiffman-Rochberg (cf. \cite{CR}, see also \cite{Ose13} for further properties). Since
\begin{equation}\label{eq:3}
MO_{S}(\phi)=\frac1{|S|} \int_S |\phi-\phi_S|=\frac2{|S|} \int_{\phi<\phi_S} (\phi_S-\phi) \leq 2\, UO_S(\phi),
\end{equation}
we see  that  bounded upper oscillation (BUO) implies BMO. One may define $BUO(\Omega,{\mathcal S})$ and $BUO_{\rm loc}(\Omega,{\mathcal S})$ analogously as the case of BMO.

Let ${\mathcal P}={\mathcal P}(\Omega)$ denote the set of relatively compact polydiscs in $\Omega$ and ${\mathcal P}_N$ the set of polydiscs $P\subset\subset \Omega$ of finite type $N$, i.e.,
$$
\max\{r_j\} \leq \min\{r_j^{1/N}\},
$$
where $N>0$ and $\{r_j\}_{1\le j\le n}$ is the\/ {\it polyradius}\/ of $P$.

 Based on Harnack's inequality and convex analysis, we are able to show  the following

  \begin{theorem}\label{th:Main1}
  \begin{enumerate}
  \item $PSH(\Omega)\subset BUO_{\rm loc}(\Omega,{\mathcal P}_N) \subset BMO_{\rm loc}(\Omega,{\mathcal P}_N)$;
  \item $PSH(\mathbb D^n)\nsubseteq BMO_{\rm loc}(\mathbb D^n,{\mathcal P})$ for $n\ge 2$, where $\mathbb D^n$ is the unit polydisc;
  \item  ${\mathcal L}\subset BUO(\mathbb C^n, \mathcal P)$; more precisely, for every $\phi \in PSH(\mathbb C^n)$ with
$$
\phi(z_1,\cdots, z_n) \leq c+\max_{1\leq j\leq n} \log (1+|z_j|), \ \  \forall \ (z_1, \cdots, z_n)\in\mathbb C^n,
$$
where $c$ is a constant, we have
$
UO_P(\phi) <  3^n
$
for all polydiscs $P$ in $\mathbb C^n$.
\end{enumerate}
\end{theorem}

For  $({\rm deg\,}p)^{-1}\log |p|\in \mathcal L$ where $p$ is a complex polynomial, we even obtain a dimension-free BUO estimate with respect to all compact convex sets.

\begin{theorem}\label{th:remez1} For every non-empty compact convex set $A$ in $\mathbb C^n$, we have
$$
UO_{A} (\log |p|)  \leq \gamma \cdot \deg p,
$$
for all $p\in\mathbb C[z_1, \cdots, z_n]$. Here the constant $\gamma\in (1,2)$ is determined by
$$
\gamma+\log(\gamma-1)=0.
$$
\end{theorem}

\textbf{Remark}: (i) The above estimate is sharp, in fact, there exists a line segment $A$ in $\mathbb C$ such that
$$UO_A(\log|z|)=\gamma.$$

(ii) In particular, if $A$ is a compact convex set in $\mathbb R^n \subset \mathbb C^n$ and all coefficients of $p$ are real, then we have
$$
UO_A (\log |p|) \leq  \gamma \cdot \deg p <2\deg p,
$$
 which is closely related the classical Remez inequality for real polynomials.  Theorem \ref{th:remez1} also suggests  to study the Remez inequality for \emph{complex} polynomials (see \cite{BG} and \cite{BJZ} for related results).

 (iii) Notice that $
1.278<\gamma<1.279.
$
By $(\ref{eq:3})$ we have
 $$
 MO_A (\log |p|) \leq  2\gamma \cdot \deg p< 2.558\cdot \deg p.
 $$
 Such dimension-free estimate (with a slightly better constant $2+\log 2\approx 2.301$) was first obtained by Nazarov et al. \cite{Nazarov}. Our proof of  Theorem \ref{th:remez1} is  elementary, however.

\medskip

For $\phi\in PSH(\Omega)$ we define the (weighted) Bergman kernel by
$$
K_{\phi,\Omega}(z)=\sup\left\{|f(z)|^2: f\in {\mathcal O}(\Omega),\  \int_\Omega |f|^2 e^{-\phi}\le 1\right\}.
$$
For a vector $a=(a_1,\cdots,a_n)$ with all $a_j>0$ we set
$$
P_{r^a}:=\{z\in \mathbb C^n: |z_j|\leq r^{a_j},\,1\le j\le n\}.
$$
It was shown in \cite{Chen} that if $\phi$ is psh on the closure of the unit ball  $\mathbb B^n$ and $a_0=(1,1/2,\cdots,1/2)$ then
$$
 \lim_{r\rightarrow 0+} \frac{\log K_{\varepsilon\phi,\mathbb B^n}(1-r,0,\cdots,0)}{\log 1/r}=n+1-\varepsilon \cdot \lim_{r\rightarrow 0+} \frac{\sup_{z\in P_{r^{a_0}}} \phi (1+z) }{\log r}
 $$
provided $\varepsilon\ll 1$, where $1+z=(1+z_1,z_2,\cdots,z_n)$. The limit in RHS of the above inequality is called the $a_0-$directional Lelong number  of $\phi$ at $(1,0,\cdots,0)$ (see \cite{K94}).

Here we will present an analogous but independent result, as an application of Theorem \ref{th:Main1}.
 For $\phi\in PSH({\mathbb D}^n)$ and $t\in \mathbb D^n$ we define
$$
\phi^t(z):=\phi(tz), \ \ tz:=(t_1z_1, \cdots, t_n z_n).
$$
A fundamental result of Berndtsson \cite{Bern06} implies that
$$
F(\phi): \, (t,z) \mapsto \log K_{\phi^t, \, \mathbb D^n}(z)
$$
is psh on $\mathbb D^n \times \mathbb D^n$.

\begin{theorem}\label{th:Main2}
For each $a=(a_1,\cdots,a_n)$ with all $a_j>0$, there exists a number $\varepsilon_0=\varepsilon_0(a,\phi,\Omega)$ such that
\begin{equation}\label{eq:5}
\lim_{r\to 0+} \frac{\sup_{t\in P_{r^a}} F(\varepsilon \phi)(t,0)/\varepsilon }{\log r}=\lim_{r\to 0+} \frac{\sup_{z\in P_{r^a}} \phi(z) }{\log r}
\end{equation}
holds for all $\varepsilon\le \varepsilon_0$.
\end{theorem}

Although Theorem \ref{th:Main2} makes sense only when $\phi$ is singular at the origin,  it is of independent interest  to study the relation between $F(\phi)$ and $\phi$ for smooth $\phi$.

\begin{theorem}\label{th:Main3}
Let $\phi$ be a smooth psh function on $\mathbb D^n$. Then
\begin{equation}\label{eq:6}
\lim_{t\to 0} \frac{\partial^2 F(\phi)}{\partial t_j \partial\bar t_k} (t,0)=
\begin{cases}
\frac12 \cdot  \frac{\partial^2 \phi}{\partial z_j \partial\bar z_j} (0), \
\text{if}\ \, j=k; \\
0, \
\text{if}\ \, j\neq k.
\end{cases}
\end{equation}
In particular $F(\phi)(t,0)$ is strictly psh at $t=0$ if $\phi$ is strictly psh at $z=0$.
\end{theorem}

\textbf{Remark}:  Since $F(\phi)(t,0)$ depends only $(|t_1|,\cdots, |t_n|)$, it follows from the psh property of $F(\phi)$ that
\begin{equation}\label{eq:bo}
\log\frac{e^{\phi(0)}}{\pi^n}=F(\phi)(0,0)\leq  F(\phi)(t,0)=\log K_{\phi^t, \,\mathbb D^n}(0).
\end{equation}
Letting $t$ tend to $(1,\cdots, 1)$, we obtain the sharp Ohsawa--Takegoshi estimate (cf. \cite{Blocki}; see also \cite{GZ,BL}):
\begin{equation}\label{eq:7}
K_{\phi,\, \mathbb D^n}(0) \geq \frac{e^{\phi(0)}}{\pi^n}.
\end{equation}
Theorem \ref{th:Main3} suggests that  one should have a better lower bound  for $K_{\phi, \,\mathbb D^n}$  in case $\phi$ is \/ {\it strictly}\/ psh.

\section{An enlightening  example}
To explain why BUO is easier than BMO, we will show that the  upper oscillation  of $\log|z|$ with respect to discs is computable. Recall that
$$
UO_{B} (\log |z|):={\sup}_B \log|z|-(\log|z|)_B
$$
for every disc $B$ in $\mathbb C$.

\begin{lemma}\label{lm:integral}
Fix $\hat{z}\in \mathbb C$ and set
$$
I(c):=\frac1{2\pi} \int_{0}^{2\pi} \log|\hat z+c e^{i\theta}| \, d\theta, \ \ c>0.
$$
Then we have
$$
I(c)=
\begin{cases}
\begin{array}{ll}
 \log|\hat z| &  \text{if}\ \  c\leq |\hat z|  \\
\log c &  \text{if}\ \  c> |\hat z|.
\end{array}
\end{cases}
$$
\end{lemma}

\begin{proof}
If $c\leq |\hat z|$ then $\log|z|$ is harmonic in the disc $\{z:|z-\hat z| < c\}$,  so that $I(c)=\log|\hat z|$, in view of the mean-value equality. For $c>|\hat z|$  we may write
$$
I(c)=\frac1{2\pi}\int_{0}^{2\pi} \log|\hat z e^{i\theta}+c|\, d\theta.
$$
As $\log|z|$ is harmonic in $\{z:|z-c|< |\hat z|\}$, we get $I(c)=\log c$.
\end{proof}

\begin{proposition}
For any disc $B$ we have
$$UO_{B} (\log |z|) \leq \log\frac{\sqrt 5+1}{2} +\frac{\sqrt 5-1}{4}.$$
Moreover, the bound is sharp.
\end{proposition}

\begin{proof}
Suppose $B=\{z:|z-\hat z|<b\}$. By Lemma \ref{lm:integral} we have
$$
(\log|z|)_{B}= \log|\hat z|, \ \ \ \ \text{if}\ \ b\leq |\hat z|,
$$
and if $b>|\hat z| $ then
\begin{eqnarray*}
(\log|z|)_{B} & = & \frac{1}{\pi b^2}\int_{0}^b 2\pi c \cdot I(c)\, dc\\
& = & \log b-\frac{1}{2} \left(1-\frac{|\hat z|^2}{b^2}\right).
\end{eqnarray*}
It follows that
$$
UO_{B} (\log |z|) =
\begin{cases}
\begin{array}{ll}
\log(b+|\hat z|)-\log |\hat z| &  \text{if} \ \ b\leq |\hat z|\\
\log(b+|\hat z|)-\log b+\frac{1}{2} \left(1-\frac{|\hat z|^2}{b^2}\right) &  \text{if}\ \ b>|\hat z| .
\end{array}
\end{cases}
$$
If $b\leq |\hat z|$ then
$$
UO_{B} (\log |z|) =\log\left(\frac b{|\hat z|}+1\right)\leq \log 2.
$$
For $b> |\hat z|$ we set $x=|\hat z|/b$ and write $UO_{B} (\log |z|)$ as
$$
 f(x)=\log(1+x)+\frac12 \cdot (1-x^2), \ \ 0<x<1.
$$
Since
$$
f'(x)=\frac1{1+x}-x,
$$
we see that $f$ is increasing on $[0, \hat x]$ and decreasing on $[\hat x, 1]$, where $
\hat x=\frac{\sqrt 5-1}{2}.
$
Notice that
$$
f(\hat x)=\log\frac{\sqrt 5+1}{2} +\frac{\sqrt 5-1}{4}.
$$
Thus $$UO_{B} (\log |z|) \leq \log\frac{\sqrt 5+1}{2} +\frac{\sqrt 5-1}{4}$$
and the equality holds if and only if
$$
\frac{|\hat z|}{b}=\frac{\sqrt 5-1}{2}.
$$
This finishes the proof.
\end{proof}

\section{Proof of Theorem \ref{th:Main1}}

\subsection{One dimensional case}
Let $\Omega$ be a domain in $\mathbb C$ and $\phi$  a subharmonic function on $\Omega$. Recall that
$$
UO_B(\phi)={\sup}_B\, \phi-\phi_B
$$
where
$
B=\{z:|z-\hat z| <r\}\subset \Omega.
$
The idea is to use \emph{Harnack's inequality} and a convexity lemma. Let us write
$$
UO_B(\phi)=I_1+I_2,
$$
where
$$
I_1={\sup}_{B} \,\phi-\phi_{\partial B}, \ \ I_2:=\phi_{\partial B}-\phi_{B},
$$
with $\phi_{\partial B}$ being the mean-value of $\phi$  over the boundary  $\partial B$. For each $\tau>0$ we  set
$$
 \tau B=\{z:|z-\hat z|<\tau r\}.
$$
 Applying Harnack's inequality to the nonpositive subharmonic function $\psi:=\phi-\sup_{B} \phi$, we get
$$
{\sup}_{\frac12 B}\, \psi = {\sup}_{\partial (\frac12 B)}\, \psi \leq \frac13 \cdot  \psi_{\partial\, B},
$$
i.e.,
$$
I_1  \leq  3\left( {\sup}_{B}\, \phi-{\sup}_{\frac12 B}\, \phi\right).
$$
Here the constant $1/3$ comes from the Poisson kernel of the unit disc since
$$
\inf_{|z|=1/2} \frac{1-|z|^2}{|1-z|^2}=\frac13.
$$
The following fact explains why we need such an estimate.

\medskip

\textbf{Fact 1}:
$
J_1:=\sup_{B} \phi- \sup_{\frac12 B} \phi
$ is continuous in $\hat z$ and $r$ respectively; moreover, it is increasing with respect to $r$.

\begin{proof} Since $\sup_{B} \phi$ is a convex function of $\log r$ (see \cite{Demailly}, Corollary 5.14), it follows that $J_1$ is a continuous increasing function of $r$. The continuity of $J_1$ in $\hat{z}$ is obvious.
\end{proof}

Let $\Omega_0$ be a relatively compact open subset in $\Omega$. Let $\delta_0$ denote the distance between $\overline{\Omega_0}$ and $\partial \Omega$. By the above fact we see that if the radius $r$ of $B\subset \Omega_0$ is less than $\delta_0/2$ then
$$
I_1\le 3\, {\sup}_{\hat{z}\in \Omega_0}\, J_1(\hat{z},\delta_0/2)<\infty,
$$
and if $r\ge \delta_0/2$ then
$$
I_1 \le 3\,{\sup}_{\Omega_0} \phi -3\, {\inf}_{\hat{z}\in \Omega_0}\, {\sup}_{\{|z-\hat{z}|<\delta_0/4\}}\,\phi<\infty.
$$

To estimate $I_2$, we need the following convexity lemma which was communicated to the second author by Bo Berndtsson:

\begin{lemma}\label{lm:convexity}
 Let $d\mu$ be a probability measure on a Borel measurable subset $S$ in $\mathbb R^n$ with barycenter $\hat t \in \mathbb R^n$. Let $f$ be a convex function on $\mathbb R^n$. Then
$$
\int_S f \, d\mu \geq f(\hat t).
$$
\end{lemma}

\begin{proof} Since $f$ is convex, there exists an affine function
$l$ such that $f(\hat t)=l(\hat t)$ and
$
f\geq l$ on $\mathbb R^n$,
which implies
$$
\int_S f\, d\mu \geq \int_S l\, d\mu = l(\hat t)=f(\hat t),
$$
where the first equality follows from the definition of  barycenter.
\end{proof}

With
$
f(t):=\phi_{\{z:|z-\hat z|=e^t r\}}
$
we have
\begin{eqnarray*}
I_2 & = & f(0)-\frac1{\pi r^2}\int_{-\infty}^0  2\pi e^{t} r\cdot f(t) \, d(e^t r)\\
& = & f(0)-\int_{-\infty}^0 f(t) \, d (e^{2t}).
\end{eqnarray*}
Since $f(t)$ is convex and $d (e^{2t})$ is a probability measure on $(-\infty, 0)$ with barycenter at $t=-1/2$, it follows from Lemma \ref{lm:convexity} that
\begin{equation}\label{eq:2.2}
\int_{-\infty}^0 f(t) \, d (e^{2t}) \geq f(-1/2),
\end{equation}
which implies
$$
I_2 \leq J_2:= f(0)-f(-1/2).
$$
Since $f$ is convex,  we get an analogous conclusion as {\bf Fact 1}:

\medskip

\textbf{Fact 2}:
$
J_2
$ is  continuous in  $\hat z$ and $r$ respectively; moreover, it is increasing with respect to $r$.

\medskip

By a similar argument as above, we may verify that
$$
{\sup}_{B\subset \Omega_0}\, I_2  <\infty.
$$

\subsection{High dimensional case}

The following result plays the role of {\bf Fact 1,2}.

\begin{lemma}\label{le:convex2}
Let $g(t)=g(t_1, \cdots, t_n)$ be a convex  function on $(-\infty, 2)^n$ which is increasing in each variable. Then
$$
{\sup}_{t\in A_N} \left[g(t)-g(t-1)\right] \leq nN \left[g(1,\cdots, 1)-g(0)\right],
$$
where $t-1:=(t_1-1, \cdots, t_n-1)$, $N\geq 1$ and
$$
A_N:=\left\{t\in (-\infty, 0]^n: \max\{-t_j\} \leq N\min\{-t_j\}\right\}.
$$
\end{lemma}

\begin{proof}
A standard regularization process reduces to the case when $\phi$ is smooth. Set
$$
f(a)=g(t_1+a,\cdots, t_n+a):=g(t+a).
$$
We have
$$
f(0)-f(-1)=\int_{-1}^0 f'(a)\, da=\int_{-1}^0 \sum g_j(t+a)\, da
$$
where $g_j:=\frac{\partial g}{\partial t_j}$.
Notice that
$$
\sum g_j(t+a)  \leq \frac{1}{\min\{-t_j-a\}}  \sum  (-t_j-a)g_j(t+a)
$$
and
$$
 \sum  (-t_j-a)g_j(-s(t+a)) =\frac{dg(-s(t+a))}{d s}
$$
is an increasing function of $s\in (-\infty, 0)$ by  convexity of $g$. Thus we have
$$
 \sum  (-t_j-a)g_j(t+a) \leq  \sum  (-t_j-a)g_j(0) \leq \max\{-t_j-a\} \sum g_j(0),
$$
which implies
$$
\sum g_j(t+a)  \leq \frac{\max\{-t_j-a\}}{\min\{-t_j-a\}}  \sum g_j(0).
$$
For any $t\in A_N$,  we have $t+a \in A_N$ (since $a\leq 0$), so that
$$
 \frac{\max\{-t_j-a\}}{\min\{-t_j-a\}} \leq N.
$$
Thus
$$
g(t)-g(t-1) \leq N \sum g_j(0).
$$
Since $g$ is convex and increasing, we have
$$
g_j(0) \leq g(1,\cdots, 1)-g(0),
$$
which finishes the proof.
\end{proof}

Let $$
P:=\left\{z\in \mathbb C^n: |z_j-\hat z_j|<r_j, \ 1\leq j\leq n\right\} \subset \Omega
$$ be a polydisc of type $N$, i.e.,
$$
\max\{r_j\} \leq \min\{r_j^{1/N}\}.
$$
Similar as above, we write
$$
UO_{P} (\phi)={\sup}_P\, \phi-\phi_P =I_1+I_2,
$$
where
$$
I_1:={\sup}_P\, \phi-\phi_{\partial P}, \ \  I_2:=\phi_{\partial P}-\phi_{P},
$$
and
$$
\partial P:=\{z\in \mathbb C^n: |z_j-\hat z_j|=r_j, \ 1\leq j\leq n\}
$$
is the \emph{Shilov boundary} of $P$. Applying Harnack's inequality  (see \cite{K94}, p.\,186) $n$-times, we get the following

\begin{lemma}\label{le3.3} $I_1\leq 3^n J_1$, where $J_1:=\sup_{P} \phi-\sup_{\frac12 P} \phi$.
\end{lemma}

Using \eqref{eq:2.2} repeatedly we get

\begin{lemma}\label{le3.4} $I_2\leq  J_2$, where $J_2:=f(0)-f(-1/2, \cdots, -1/2)$ with
$$
f(t):=\phi_{\{z\, :\, |z_j-\hat z_j|=e^{t_j} r_j,\ 1\le j\le n\}}.
$$
\end{lemma}

Since both $\sup_{P} \phi$ and $\phi_{\partial P}$ are continuous in $\hat z_j$ and convex increasing with respect to $\log r_j$ for all $j$, it follows from Lemma \ref{le:convex2} (through a similar argument as the one-dimensional case) that
$$
{\sup}_{P\in {\mathcal P}_N|_{\Omega_0}
} \left(J_1+J_2\right) <\infty,
$$
for every  open set $\Omega_0\subset\subset\Omega$, which finishes the proof of the first part of Theorem \ref{th:Main1}.

\subsection{A counterexample} For the second part of Theorem \ref{th:Main1}, we need to construct a counterexample. For the sake of simplicity, we only consider the case  $n=2$. It suffices to verify the following

\begin{theorem}
\label{th:example}
Set
$
\phi(z, w):=-\sqrt{(\log|z|+\log|w|)\log|w|},
$
 $z,w\in\mathbb D$. Then we have $\phi\in PSH(\mathbb D^2)$, while
$$
\sup_{0<r_1, r_2<1} \frac{1}{|\mathbb D^2_r|}\int_{\mathbb D^2_r} |\phi-\phi_{\mathbb D^2_r}| =\infty,
$$
where
$$
\mathbb D^2_r:=\left\{(z,w)\in\mathbb C^2: |z|<r_1, \, |w|<r_2 \right\}.
$$
\end{theorem}

The following lemma shows that {\bf Fact 1, 2}  is no more true for general bidiscs.

\begin{lemma}
\label{le:2.5}
$f(x,y):=-\sqrt{(x+y)y}$ is  convex  on $ (-\infty, 0)^2$ and increasing in each variable; moreover,
\begin{equation}\label{eq:2.4}
{\sup}_{\{x, y\leq -1\}} \left[f(x,y)-f(x-1, y-1)\right]=\infty.
\end{equation}
\end{lemma}

\begin{proof}
The first conclusion follows by a straightforward calculation. For \eqref{eq:2.4} it suffices to note that
$$
f(x,-1)-f(x-1, -2)=\frac{5-x}{\sqrt{6-2x}+\sqrt{1-x}} \to \infty
$$
as  $x\to-\infty$. The proof is complete.
\end{proof}

Let us first verify that $\phi \notin BUO_{\rm loc}(\mathbb D^2, \mathcal P)$.

\begin{lemma}\label{lm:NonBUO}
$\sup_{0<r_1, r_2<1} \sup_{\mathbb D^2_r}(\phi- \phi_{\mathbb D^2_r})=\infty$.
\end{lemma}

\begin{proof} With $x=\log r_1$ and $y=\log r_2$, we get
$$
{\sup}_{\mathbb D^2_r} \left(\phi- \phi_{\mathbb D^2_r}\right)= f(x,y)-\int_{-\infty}^0 \int_{-\infty}^0  f(x+t, y+s) \, de^{2t} d e^{2s}=: I(x,y).
$$
Integrate by parts with respect to $t$ and $s$ successively, we may write
$$
I(x,y)=I_1 +I_2,
$$
where
$$
I_1=\int_{-\infty}^0 \frac{x+2y+2s}{-4f(x+t, y+s)} \, de^{2s}
$$
and
$$
I_2=\int_{(-\infty, 0)^2} \frac{y+s}{-4f(x+t, y+s)} \, de^{2t} d e^{2s}.
$$
Obviously,  $I_2(x,-1)$ is bounded on $(-\infty, 0]$, but  $I_1(x,-1)\to \infty$  as $x\to -\infty$, from which the assertion immediately follows.
\end{proof}

\begin{proof}[Proof of Theorem \ref{th:example}] By Lemma \ref{lm:convexity} we have (still with $x=\log r_1,\,y=\log r_2$)
\begin{eqnarray*}
\phi_{\mathbb D^2_r} & = & \int_{-\infty}^0 \int_{-\infty}^0 f(x+t, y+s) \, de^{2t} d e^{2s} \\
& \geq & f(x-1/2, y-1/2)= {\sup}_{\mathbb D^2_{e^{-1/2} r}}\,\phi,
\end{eqnarray*}
which yields
\begin{eqnarray*}
\frac{1}{|\mathbb D^2_r|}\int_{\mathbb D^2_r} |\phi-\phi_{\mathbb D^2_r}|
& \geq &
\frac{1}{|\mathbb D^2_r|}\int_{\mathbb D^2_{e^{-1/2} r}}   \left({\sup}_{\mathbb D^2_{e^{-1/2} r}}\,\phi - \phi\right)\\
 & = & e^{-2}  \left(  { \sup}_{\mathbb D^2_{e^{-1/2} r}}\,\phi- \phi_{\mathbb D^2_{e^{-1/2} r}}    \right).
\end{eqnarray*}
By a similar argument as  Lemma \ref{lm:NonBUO}, we conclude the proof of Theorem \ref{th:example}.
\end{proof}

\subsection{Lelong class}

In this section we shall prove the third part of Theorem \ref{th:Main1}. The key ingredient is the following counterpart of Lemma \ref{le:convex2}.

\begin{lemma}\label{le:convex3}
Let $g(t)=g(t_1, \cdots, t_n)$ be a convex  function on $\mathbb R^n$ which is increasing in each variable. Assume that
$$
g(t) \leq \max_{1\leq j\leq n}\{\log(1+e^{t_j})\}, \ \ \forall \ t\in\mathbb R^n.
$$
Then for every $M>0$ we have
$$
{\sup}_{t\in \mathbb R^n} \left[g(t)-g(t-M)\right] \leq M,
$$
where $t-M:=(t_1-M, \cdots, t_n-M)$.
\end{lemma}

\begin{proof} For fixed $t$, we consider the following convex increasing function
$$
f(s):=g(t_1+s, \cdots, t_n+s)
$$
on $\mathbb R$. Convexity of $f$ gives
$$
\frac{f(0)-f(-M)}M \leq \lim_{s\to \infty} \frac{f(s)-f(0)}{s}.
$$
By the assumption, we have
$$
f(0) \leq f(s) \leq  \max_{1\leq j\leq n}\{\log(1+e^{t_j} e^s)\}
$$
for every $s\geq 0$, so that
$$
\lim_{s\to \infty} \frac{f(s)-f(0)}{s}\leq 1.
$$
The proof is complete.
\end{proof}

\begin{proof}[Proof of  the third part of Theorem \ref{th:Main1}]
Again for any polydisc
$$
P:=\left\{z\in \mathbb C^n: |z_j-\hat z_j|<r_j, \ 1\leq j\leq n\right\},
$$
we may write
$$
UO_{P} (\phi)={\sup}_P\, \phi-\phi_P =I_1+I_2,
$$
where
$$
I_1:={\sup}_P\, \phi-\phi_{\partial P}, \ \  I_2:=\phi_{\partial P}-\phi_{P}.
$$
By Lemma \ref{le3.3} we have
$$
I_1\leq 3^n \left({\sup}_P\, \phi-{\sup}_{\frac12 P}\, \phi\right).
$$
Put
$$
P_t:=\left\{z\in \mathbb C^n: |z_j-\hat z_j|<e^{t_j}r_j, \ 1\leq j\leq n\right\}
$$
and $f_1(t):=\sup_{P_t} \phi$.
Since $\phi\in {\mathcal L}$,  we know that for some constant $c_1\gg 1$  the function $f_1-c_1$ satisfies the assumption in Lemma \ref{le:convex3}, so that
$$
{\sup}_{P}\, \phi-{\sup}_{\frac12 P}\,\phi=f_1(0)-f_1(-\log 2) \leq \log 2,
$$
which in turn implies
$$
I_1\leq 3^n \log 2.
$$
Moreover, we infer from Lemma \ref{le3.4} that
$$
I_2\leq f(0)-f(-1/2, \cdots, -1/2),  \ \  f(t):=\phi_{\partial P_t}.
$$
Applying Lemma \ref{le:convex3} in a similar way as above, we have
$$
I_2 \leq 1/2.
$$
Thus
$$
UO_{P} (\phi) \leq 3^n \log 2+ 1/2 < 3^n,
$$
which finishes the proof.
\end{proof}

\section{Proof of Theorem \ref{th:remez1}}

  The starting point is the following

\begin{definition}[$\gamma$-constant] We shall define the constant $\gamma$ as the BUO norm of\/ $\log |z|$ on $\mathbb C$ with respect to all line segments. More precisely,
$$
\gamma:=\sup_{a\neq b\in \mathbb C} UO_{[a,b]} (\log |z|),
$$
where $[a,b]$ denotes the line segment connecting $a$ and $b$,  and the upper oscillation is defined by
$$
UO_{[a,b]} (\log |z|) :=\left({\sup}_{0\leq t\leq 1}\,\log |a(1-t)+bt| \right)- \int_{0}^1  \log |a(1-t)+bt| \, dt.
$$
\end{definition}

The key step is to show the following

\begin{lemma}\label{lm:gamma}
$1<\gamma<2$ is determined by
\begin{equation}\label{eq:gamma}
\gamma+\log(\gamma-1)=0.
\end{equation}
\end{lemma}

\begin{proof}

For each pair $a, b \in \mathbb C$, we shall compute
$$
UO_{[a,b]} (\log |z|)={\sup}_{[a,b]}\, \log |z|-(\log |z|)_{[a,b]}.
$$
Since $\log |z|$ is $S^1$-invariant, by a rotation of $z$, we may assume that
$$
b \in \mathbb R,\ \ \    b>|a|.
$$
Thus
$$
{\sup}_{[a,b]}\, \log |z|=\log b
$$
is independent of $a$.
Since
\begin{eqnarray*}
(\log |z|)_{[a,b]} & = & \int_{0}^1  \log |a(1-t)+bt| \, dt\\
& \ge & \int_{0}^1  \log |{\rm Re\,}a\cdot(1-t)+bt| \, dt\\
& = & (\log |z|)_{[{\rm Re\,}a,b]}
\end{eqnarray*}
with equality holds if and only if $a\in \mathbb R$. Thus it  suffices to verify \eqref{eq:gamma} for
$$
a, b\in \mathbb R, \ \ \ |a|<b.
$$
Consider $\log|z|-\log b$ instead of $\log|z|$, one may further assume that
$$
b=1, \ \ \ -1<a<1,
$$
which implies
$$
UO_{[a,1]} (\log|z|)= \log 1- (\log|z|)_{[a, 1]}=-(\log|z|)_{[a, 1]}.
$$
We divide into two case.
(i) $0\leq a<1$.  Then we have
$$
-(\log|z|)_{[a, 1]}=\frac{-1}{1-a} \int_{a}^1 \log x \, dx=\frac{a\log a}{1-a}+1 \leq 1.
$$
(ii) $-1<a<0$. Then we have
$$
-(\log|z|)_{[a, 1]}=\frac{-1}{1-a} \int_{a}^1 \log |x| \, dx=\frac{a\log (-a)}{1-a}+1 >1.
$$
Thus
$$
\gamma=\sup_{-1<a<0} \frac{a\log (-a)}{1-a}+1.
$$
It suffices to verify that $\gamma$ satisfies $(\ref{eq:gamma})$.
To see this,  put
$$
t^{-1}:=1-a \in (1,2)
$$
and write
$$
\frac{a\log (-a)}{1-a}= (1-t)\log t-(1-t)\log (1-t)=:f(t).
$$
Since
$$
f'(t)=t^{-1}-\log t +\log(1-t),
$$
it follows that $f'(t)=0$ if and only if
$$
t^{-1}=\log \frac{1}{t^{-1}-1},
$$
i.e.,
$$
1-a+\log(-a)=0.
$$
Thus we have
$$
\gamma=\sup_{-1<a<0} \frac{a\log (-a)}{1-a}+1= \frac{a_0\log (-a_0)}{1-a_0}+1,
$$
where $a_0$ is determined by
\begin{equation}\label{eq:a_0}
1-a_0+\log(-a_0)=0,
\end{equation}
which gives
$$
\gamma=1-a_0 \in (1,2).
$$
It is clear that $(\ref{eq:a_0})$ is equivalent to $(\ref{eq:gamma})$.
\end{proof}

 Since a translation of a line segment is still a line segment, we know that $\log|z-z_0|$ and $\log|z|$ have the same line segment BUO norm. This fact can be used to estimate  the line segment BUO norm of $\log|p|$ for general polynomials $p$. In fact, if we write
$$
p=a_0(z-a_1)^{n_1} \cdots (z-a_k)^{n_k},
$$
then
$$
{\sup}_{[a,b]}\, \log|p| \leq \log |a_0|+ \sum_{j=1}^k n_j {\sup}_{[a,b]}\, \log|z-a_j|
$$
and
$$
(\log|p|)_{[a,b]}=  \log |a_0|+ \sum_{j=1}^k n_j (\log|z-a_j|)_{[a,b]}.
$$
Thus
\begin{eqnarray*}
UO_{[a,b]} (\log|p|) & := & {\sup}_{[a,b]}\, \log|p|-(\log|p|)_{[a,b]}\\
& \leq & \sum_{j=1}^k n_j  UO_{[a,b]}(\log|z-a_j|).
\end{eqnarray*}
This combined with the fact $UO_{[a,b]}(\log|z-a_j|) \leq \gamma$  gives
\begin{equation}\label{eq:ss}
UO_{[a,b]} (\log |p|)  \leq \gamma \cdot \deg p.
\end{equation}
for all polynomials $p$ and all $a,b\in\mathbb C$.

Now we may conclude the proof of Theorem \ref{th:remez1} as follows.
 Since $A$ is compact, we may choose $z_0\in A$ such that
$$
|p(z_0)|={\sup}_{z\in A}\, |p(z)|.
$$
For every ray (half line), say $L$, starting from $z_0$, we see that $A\cap L$ is a line segment in view of convexity of  $A$. Let $L_{\mathbb C}$ be the complex line containing $L$. Apply \eqref{eq:ss} to $p|_{L_{\mathbb C}}$, we have
$$
UO_{A\cap L} (\log |p|)= UO_{A\cap L} (\log |p|_{L_{\mathbb C}}|) \leq \gamma\deg p|_{L_{\mathbb C}}\leq \gamma \deg p,
$$
which gives
$$
UO_{A} (\log |p|)  \leq \gamma \deg p
$$
since $UO_{A} (\log |p|) $ is a certain average of $UO_{A\cap L} (\log |p|)$ for all $L$ starting from $z_0$: in fact, since $z_0$ is a maximum point of $\log|p|$ on $A$ and $L$ contains $z_0$, we always have
$$
\sup_{A\cap L} \log|p|= \log|p(z_0)|,
$$
together with \eqref{eq:ss} it gives
\begin{equation}\label{eq:last}
\gamma\cdot \deg p\geq UO_{A\cap L} (\log |p|) = \log|p(z_0)| - \frac1{|A\cap L|} \int_{A\cap L} \log |p|.
\end{equation}
Thus
\begin{eqnarray*}
\int_A \log |p| & = & \int_{ S_{2n-1}}   \int_{A\cap L}  \log |p| \, d\mu(L)\\ 
& \ge & (\log|p(z_0)|-\gamma\cdot \deg p)\int_{ S_{2n-1}}  |A\cap L| \, d\mu(L)\\
& = & (\log|p(z_0)|-\gamma\cdot \deg p) |A|,
\end{eqnarray*}
where $d\mu$ is a certain measure on the unit sphere $S_{2n-1}$ and we identify the set of rays $L$ starting from $z_0$ with $S_{2n-1}$. Notice that the above inequality gives
$$
\gamma\cdot \deg p \geq \log|p(z_0)| - \frac1{|A|} \int_A \log |p|=UO_A (\log|p|),
$$
from which the assertion immediately follows.

\section{Proof of Theorem \ref{th:Main2}}

The starting point is the following

\begin{proposition}[John-Nirenberg inequality]\label{prop:J-N}
Suppose $\phi\in PSH(\Omega)$ and  $\Omega_0\subset\subset\Omega$ is open. For each  $a=(a_1,\cdots,a_n)$ with all $a_j>0$ there exists $\varepsilon_0=\varepsilon(a,\phi,\Omega_0,\Omega) >0$ such that
$$
\sup_{ P_{r^a}(\hat{z})\subset \Omega_0}  \frac{1}{|P_{r^a}(\hat{z})|}\int_{P_{r^a}(\hat{z})} e^{-\varepsilon(\phi-\sup_{P_{r^a}(\hat{z})} \phi)} <\infty,
$$
for every $\varepsilon \leq \varepsilon_0$. Here
$$
P_{r^a}(\hat{z})=\{z\in \mathbb C^n: |z_j-\hat{z}_j|\leq r^{a_j}\}.
$$
\end{proposition}

Although the argument is fairly standard, we will provide a proof in Appendix, because the result cannot be found in literature explicitly.

\begin{lemma}\label{lm:BergmanEstimate}
Let $\psi$ be a psh function on $\Omega$ which satisfies  $\sup_{\Omega} \psi <\infty$ and $\int_\Omega e^{-\psi} <\infty$. Suppose $\Omega$ is circular, i.e., $\zeta z\in \Omega$ for every $\zeta\in \mathbb C$, $|\zeta|\leq 1$, and $z\in \Omega$. Then
\begin{equation}\label{eq:BergmanEstimate}
\left(\frac{1}{|\Omega|} \int_{\Omega}  e^{-(\psi-\sup_{\Omega} \psi)} \right)^{-1} \leq K_{\psi, \,\Omega}(0) \cdot |\Omega|\cdot e^{-\sup_\Omega \psi} \leq 1.
\end{equation}
\end{lemma}

\begin{proof}
The extremal  property of the Bergman kernel implies that
$$
K_{\psi, \,\Omega}(0)  \geq \frac{1}{\int_{\mathbb D_r} e^{-\psi}}
$$
and the first inequality in (\ref{eq:BergmanEstimate}) holds.
On the other hand, as $\Omega$ is circular, it is easy to verify that
$$
f(0)=\frac{1}{|\Omega|} \int_{\Omega} f
$$
for all $f\in {\mathcal O}(\Omega)$. Thus we have
$$
|f(0)|^2 = \left|\,\frac{1}{|\Omega|} \int_{\Omega} f\,\right|^2 \leq \frac{1}{|\Omega|}  \int_{\Omega}  |f|^2e^{-\psi} \cdot  e^{\sup_\Omega \psi},
$$
so that the second inequality in (\ref{eq:BergmanEstimate}) also holds.
\end{proof}

\begin{proof}[Proof of Theorem \ref{th:Main2}]
Since
$$
\int_{\mathbb D^n} |f|^2e^{-\varepsilon \phi^t} =|t_1\cdots t_n|^{-2}\int_{\mathbb D_t^n} |f|^2e^{-\varepsilon \phi},\ \ \ \forall\,f\in {\mathcal O}(\Omega),
$$
it follows that
\begin{equation}\label{eq:2.5}
K_{\varepsilon \phi^t, \,\mathbb D^n}(z)=\frac{|\mathbb D^n_t|}{|\mathbb D^n|} \cdot K_{\varepsilon\phi, \,\mathbb D^n_t}(z),\end{equation}
where
$$
\mathbb D^n_t:=\left\{z\in {\mathbb C}^n: |z_j|<|t_j|, \ 1\le j\le n \right\}.
$$
Thus we have
$$
F(\varepsilon \phi)(t,0)=\log (|\mathbb D^n_t|\cdot K_{\varepsilon\phi, \,\mathbb D^n_t}(0))-n \log \pi.
$$
This combined with Lemma \ref{lm:BergmanEstimate} gives
$$
-\log \left(\frac{1}{|\mathbb D^n_t|} \int_{\mathbb D^n_t}  e^{-\varepsilon(\phi-\sup_{\mathbb D^n_t} \phi)}\right)-n \log \pi
 \leq F(\varepsilon \phi)(t,0) - \varepsilon\, {\sup}_{\mathbb D^n_t}\, \phi \leq -n \log \pi.
$$
By Proposition \ref{prop:J-N}, we conclude the proof.
\end{proof}

\section{Proof of Theorem \ref{th:Main3}}
Recall that
$$
\phi^t(z):=\phi(t_1z_1,\cdots, t_nz_n).
$$
By Proposition 2.2 in \cite{Bern17}, we have
\begin{equation}\label{eq:2.6}
\frac{\partial}{\partial t_j}  K_{\phi^t, \,\mathbb D^n}(0)=\int_{\mathbb D^n} \frac{\partial \phi^t}{\partial t_j}\,|K_{\phi^t, \,\mathbb D^n}(z,0)|^2 e^{-\phi^t},
\end{equation}
where  $K_{\phi^t, \,\mathbb D^n}(z,0)$  satisfies the following reproducing property
$$
f(0)=\int_{\mathbb D^n } f(z)\overline{K_{\phi^t, \,\mathbb D^n}(z,0)}e^{-\phi^t}
$$
for all $L^2$ holomorphic functions $f$ on $\mathbb D^n$. In particular, if $f=zK_{\phi^t, \,\mathbb D^n}(z,0)$ then
$$
0=\int_{\mathbb D^n} z\cdot |K_{\phi^t, \,\mathbb D^n}(z,0)|^2 e^{-\phi^t},
$$
and since
$
\frac{\partial \phi^t}{\partial t_j}|_{t=0}=z_j\phi_{z_j}(0),
$
we get
$$
\left.\int_{\mathbb D^n} \frac{\partial \phi^t}{\partial t_j}\right|_{t=0} |K_{\phi^t, \,\mathbb D^n}(z,0)|^2 e^{-\phi^t}=0
$$
for all $t\in\mathbb D^n$. Thus we may write \eqref{eq:2.6} as
$$
\frac{\partial}{\partial t_j} K_{\phi^t, \,\mathbb D^n}(0)
=\int_{\mathbb D^n} \left(\frac{\partial \phi^t}{\partial t_j}-\left.\frac{\partial \phi^t}{\partial t_j}\right|_{t=0}\right)|K_{\phi^t, \,\mathbb D^n}(z,0)|^2 e^{-\phi^t}.
$$
In particular,
$$
\left.\frac{\partial}{\partial t_j} K_{\phi^t, \,\mathbb D^n}(0)\right|_{t=0}=0.
$$
Thus we can further write \eqref{eq:2.6} as
$$
\frac{\partial}{\partial t_j} K_{\phi^t, \,\mathbb D^n}(0)-\left.\frac{\partial}{\partial t_j} K_{\phi^t, \,\mathbb D^n}(0)\right|_{t=0}=\int_{\mathbb D^n} \left(\frac{\partial \phi^t}{\partial t_j}-\left.\frac{\partial \phi^t}{\partial t_j}\right|_{t=0}\right)|K_{\phi^t, \,\mathbb D^n}(z,0)|^2 e^{-\phi^t},
$$
which implies
$$
\left.\frac{\partial^2}{\partial t_j \partial \bar t_k} K_{\phi^t, \,\mathbb D^n}(0)\right|_{t=0}=\int_{\mathbb D^n} \left.\frac{\partial^2 \phi^t}{\partial t_j\partial \bar t_k}\right|_{t=0}\cdot |K_{\phi(0), \,\mathbb D^n}(z,0)|^2 e^{-\phi(0)}.
$$
Since
$$
K_{\phi(0), \,\mathbb D^n}(z,0)=\frac{e^{\phi(0)}}{\pi^n}
$$
and
$$
\left.\frac{\partial^2 \Phi}{\partial t_j\partial \bar t_k}\right|_{t=0}=z_j\bar z_k \phi_{z_j \bar z_k}(0),
$$
we get
$$
\left.\frac{\partial^2}{\partial t_j \partial \bar t_k} K_{\phi^t, \,\mathbb D^n}(0)\right|_{t=0}=\frac{e^{\phi(0)}\phi_{z_j \bar z_k}(0)}{\pi^{2n}}\int_{\mathbb D^n} z_j \bar z_k.
$$
Notice that
$$
\int_{\mathbb D^n} z_j \bar z_k=
\begin{cases}
\begin{array}{ll}
\pi^{n}/2 & \text{if}\ \, j=k\\
0 & \text{if}\ \, j\neq k,
\end{array}
\end{cases}
$$
and
$$
\left.\frac{\partial^2}{\partial t_j \partial \bar t_k} K_{\phi^t, \,\mathbb D^n}(0)\right|_{t=0}
= K_{\phi^t, \,\mathbb D^n}(0)\cdot \left.\frac{\partial^2}{\partial t_j \partial \bar t_k} \log K_{\phi^t, \,\mathbb D^n}(0)\right|_{t=0},
$$
our assertion follows.

\section{Appendix}

In this section we provide a proof of Proposition \ref{prop:J-N}. Let us first recall a few basic facts in real-variable theory, by following Stein \cite{Stein}. A\/ {\it quasi-distance} \/ defined on ${\mathbb R}^m$ means a nonnegative continuous function $\rho$ on ${\mathbb R}^m\times
\mathbb R^m$ for which there exists a constant $c>0$ such that
\begin{enumerate}
\item $\rho(x,y)=0$ iff $x=y$;
\item $\rho(x,y)\le c\rho(y,x)$;
\item $\rho(x,y)\le c(\rho(x,z)+\rho(y,z))$.
\end{enumerate}
Given such a $\rho$, we define "balls"
$$
B(x,r):=\{y:\rho(y,x)<r\},\ \ \ r>0.
$$
One can verify that there exists a constant $c_1>1$ such that for all $x,y$ and $r$,
\begin{equation}\label{eq:App1}
B(x,r)\cap B(y,r)\neq  \emptyset\ \ \  \Rightarrow \ \ \ B(y,r)\subset B(x,c_1r).
\end{equation}
In the case of Proposition \ref{prop:J-N}, we define
$$
\rho(z,w)=\max_k |z_k-w_k|^{1/a_k},\ \ \ z,w\in {\mathbb C}^n.
$$
It is easy to verify that $\rho$ is a quasi-distance on ${\mathbb C}^n$ and
$$
B(\hat{z},r)=P_{r^a}(\hat{z}),\ \ \ \hat{z}\in \mathbb C^n,\, r>0.
$$
Besides (\ref{eq:App1}), the following properties also hold for $B(\hat{z},r)$:
\begin{equation}\label{eq:App2}
|B(\hat{z},c_1r)|\le c_1^{2\sum_k a_k} \cdot |B(\hat{z},r)|=:c_2\, |B(\hat{z},r)|;
\end{equation}
\begin{equation}\label{eq:App3}
{\bigcap}_r \, \overline{B}(\hat{z},r)=\{\hat{z}\}\ \ \ {\text{and}\ \ \ }  {\bigcup}_r\, B(\hat{z},r)=\mathbb C^n;
\end{equation}
\begin{equation}\label{eq:App4}
\text{For each open set\/ $U$\/ and each $r>0$, the function $\hat{z}\mapsto |B(\hat{z},r)\cap U|$ is continuous.}
\end{equation}
Fix a pair of positive constants $c^\ast$ and $c^{\ast\ast}$ with $1<c^\ast<c^{\ast\ast}$. For $B=B(\hat{z},r)$ we define $B^\ast=B(\hat{z},c^\ast r)$ and $B^{\ast\ast}=B(\hat{z},c^{\ast\ast} r)$. Then we have

\begin{lemma}[cf. \cite{Stein}, p.\,15--16]\label{lm:balls}
Choose $c^\ast=4c_1^2$ and $c^{\ast\ast}=16c_1^2$. Given a closed nonempty set $F\subset {\mathbb C}^n$, there exists a collection of balls $\{B_k\}$ such that
\begin{enumerate}
\item The $B_k$ are pairwise disjoint;
\item $\bigcup_k B_k^\ast = F^c:=\mathbb C^n\backslash F$;
\item $B_k^{\ast\ast}\cap F\neq \emptyset$ for each $k$.
\end{enumerate}
\end{lemma}

 \begin{proposition}[Calder\'on-Zygmund decomposition]\label{prop:C-Z}
Let $B_0$ be a ball in $\mathbb C^n$ and  $f\in L^1(B_0)$. There is a constant $c=c(c_1,c_2)>0$ such that given a positive number $\alpha$, there exists  a sequences of
balls $\{B_k^\ast\}$ in $B_0$ such that
\begin{enumerate}
\item $|f(z)|\le \alpha$, for a.e. $z\in B_0\backslash\bigcup_k B_k^\ast$;
\item $
\int_{B_k^\ast} |f|\le c\alpha |B_k^\ast|,
$
for each $k$;
\item $\sum_k |B_k^\ast| \le \frac{c}{\alpha} \int_{B_0} |f|$.
\end{enumerate}
\end{proposition}

\begin{proof}
We extend $f$ to an integrable function on ${\mathbb C}^n$ by setting $f=0$ outside $B_0$.
Recall  the following two types of Hardy-Littlewood maximal functions:
$$
Mf(z):=\sup_{r>0}\,  \frac1{|B(z,r)|} \int_{B(z,r)} |f|,
$$
$$
\widetilde{M}f(z):=\sup_{z\in B}\, \frac1{|B|} \int_B |f|
$$
where the supremum is taken over all balls $B$ containing $z$. The relationship between $Mf$ and $\widetilde{M}f$ is as follows:
\begin{equation}\label{eq:compare}
Mf \le \widetilde{M}f\le c_2\, Mf.
\end{equation}
Notice that
$$
E_\alpha:=\left\{z\in B_0: \widetilde{M}f(z)>\alpha\right\}
$$
is an open set  since $\widetilde{M}f$ is lower semicontinuous, and
$$
|E_\alpha|\le \frac{c}{\alpha} \int_{\mathbb C^n}|f| =  \frac{c}{\alpha} \int_{B_0}|f|
$$
in view of (\ref{eq:compare}) and \cite{Stein}, p. 13, Theorem 1.  Here and in what follows $c$ will denote a generic positive constant depending only on $c_1,c_2$. With $F:=\mathbb C^n\backslash E_\alpha$  we choose balls $\{B_k\}$, $\{B_k^\ast\}$ and $\{B_k^{\ast\ast}\}$ according to Lemma \ref{lm:balls}. Then we have
 $$
 {\sum}_k |B_k^\ast| \le  c \, {\sum}_k  |B_k| \le c\,|E_\alpha| \le \frac{c}{\alpha} \int_{B_0} |f|.
  $$
  Since $B_k^{\ast\ast}\cap F\neq \emptyset$ for each $k$, we have
  $$
  \int_{B_k^\ast} |f|\le  \int_{B_k^{\ast\ast}} |f| \le \alpha |B_k^{\ast\ast}| \le c\alpha |B_k^\ast|.
      $$
      Finally, by (\ref{eq:compare}) and \cite{Stein}, p.\,13, Corollary, we know that  $|f(z)|\le \widetilde{M}f(z)$ for a.e. $z$, from which  {\it (1)} immediately follows.
  \end{proof}

\begin{proof}[Proof of Proposition \ref{prop:J-N}]
By Theorem \ref{th:Main1}, we know that
$$
M:=\sup_{B(\hat{z},r)\subset \Omega_0} \frac1{|B(\hat{z},r)|} \int_{B(\hat{z},r)} |\phi-\phi_{B(\hat{z},r)}|<\infty.
$$
Assume without loss of generality $M=1$.  Fix a ball $B_0\subset \Omega_0$. It suffices to show
\begin{equation}\label{eq:6.6}
|\{z\in B_0: |\phi-\phi_{B_0}|>t\}|\le {\rm const}\cdot e^{-\varepsilon t}\,|B_0|,\ \ \ t>0,
\end{equation}
for certain $\varepsilon\ll 1$.
With $c$ as Proposition \ref{prop:C-Z} we choose
$$
\alpha>c>1\ge \frac1{|B_0|}\int_{B_0} |\phi-\phi_{B_0}|.
$$
Applying Proposition \ref{prop:C-Z} with
$
f=
|\phi-\phi_{B_0}|,
$
 we have a sequence of balls $\{B_k^{(1)}\}$ in $B_0$ such that
$$
|\phi(z)-\phi_{B_0}|\le \alpha \ \ \ \text{a.e.}\ \ \ z\in B_0\backslash {\bigcup}_k B_k^{(1)},
$$
$$
{\sum}_k |B_k^{(1)}|\le \frac{c}{\alpha} \int_{B_0} |\phi-\phi_{B_0}|\le  \frac{c}{\alpha}\,|B_0|,
$$
and
$$
|\phi_{B_k^{(1)}}-\phi_{B_0}|\le \frac1{|B_k^{(1)}|} \int_{B_k^{(1)}} |\phi-\phi_{B_0}|\le c\alpha.
$$
Applying Proposition \ref{prop:C-Z} with
$
f=
|\phi-\phi_{B_k^{(1)}}|
$
for each $k$, we obtain a sequence of balls $\{B_k^{(2)}\}$ in $\bigcup_k B_k^{(1)}$ such that
$$
{\sum}_k |B_k^{(2)}|\le \frac{c}{\alpha} \,{\sum}_k \int_{B_k^{(1)}} |\phi-\phi_{B_k^{(1)}}|\le \frac{c}{\alpha} \,{\sum}_k
|B_k^{(1)}| \le \frac{c^2}{\alpha^2} \, |B_0|
$$
and
$$
|\phi(z)-\phi_{B_k^{(1)}}|\le \alpha \ \ \ \text{a.e.}\ \ \ z\in B_k^{(1)}\backslash {\bigcup}_k B^{(2)}_k,
$$
which in turn implies
$$
|\phi(z)-\phi_{B_0}|\le 2\cdot c\alpha \ \ \ \text{a.e.}\ \ \ z\in B_0\backslash {\bigcup}_k B^{(2)}_k.
$$
Continue this process. For each $j$ there exists a sequence of balls $\{B_k^{(j)}\}$ in $\bigcup_k B_k^{(j-1)}$ such that
$$
{\sum}_k |B_k^{(j)}|\le \frac{c^j}{\alpha^j} \, |B_0|,
$$
$$
|\phi(z)-\phi_{B_0}|\le j\cdot c\alpha \ \ \ \text{a.e.}\ \ \ z\in B_0\backslash {\bigcup}_k B^{(j)}_k.
$$
Thus
$$
|\{z\in B_0:|\phi-\phi_{B_0}|>j\cdot c\alpha\}|\le {\sum}_k |B_k^{(j)}|\le \frac{c^j}{\alpha^j} \, |B_0|.
$$
For any $t$ there exists an integer $j$ such that $t\in [j\cdot c\alpha,(j+1)\cdot c\alpha)$. It follows that
$$
(c/\alpha)^j=(\alpha/c)e^{-(j+1)\log \alpha/c}\le (\alpha/c)e^{-\frac{\log \alpha/c}{c\alpha} \, t},
$$
from which \eqref{eq:6.6} immediately follows. Now we have
$$
\frac{1}{|B_0|} \int_{B_0} e^{\varepsilon |\phi-\phi_{B_0}|}  = \frac{1}{|B_0|} \int_{0}^\infty  e^{\varepsilon t} |\{|\phi-\phi_{B_0}|>t\}| +  \frac{ |\{\phi=\phi_{B_0}\}| }{|B_0|} \leq  {\rm const}+1,
$$
which gives
$$
\frac{1}{|B_0|} \int_{B_0} e^{-\varepsilon (\phi-\sup_{B_0}\phi)}  \leq ({\rm const}+1) e^{\varepsilon(\sup_{B_0}\phi-\phi_{B_0})},
$$
By Theorem 1.1, $\sup_{B_0}\{\sup_{B_0}\phi-\phi_{B_0}\} <\infty$, thus Proposition \ref{prop:J-N} follows.
\end{proof}

\subsection*{Acknowledgments} The authors  would like to thank Ahmed Zeriahi for bringing their attention to  the reference \cite{BJZ}.
The second author would like to thank Bo Berndtsson for numerous useful discussions about the topics of this paper.

\end{document}